\documentclass{amsart}
\usepackage{amssymb,
amsfonts,
enumitem,
mathtools,
tikz,
upgreek,
verbatim,
hyphenat,
url,
}
\usepackage[T1]{fontenc}
\usepackage[shadow]{todonotes}

\usepackage{newpxtext,newpxmath}
\usepackage{subfig}
\usetikzlibrary{calc}

\usetikzlibrary{shapes.misc, positioning}

\usepackage[colorlinks=true,linkcolor={black},citecolor={blue}]{hyperref}%

\newcommand{\markdef}[1]{\textbf{#1}}

\newcommand{\lh}{\operatorname{lh}}

\newcommand{\Gen}{\operatorname{Gen}}
\newcommand{\cat}{^\frown}

\DeclareMathOperator{\Aut}{Aut}

\DeclareMathOperator{\Add}{Add}
\DeclareMathOperator{\graph}{graph}

\newenvironment{enumerate-(a)}{\begin{enumerate}[label={\upshape (\alph*)}, leftmargin=2pc]}{\end{enumerate}}
\newenvironment{enumerate-(a)-r}{\begin{enumerate}[label={\upshape (\alph*)}, leftmargin=2pc,resume]}{\end{enumerate}}
\newenvironment{enumerate-(a)-5}{\begin{enumerate}[label={\upshape (\alph*)}, leftmargin=2pc,start=5]}{\end{enumerate}}
\newenvironment{enumerate-(A)}{\begin{enumerate}[label={\upshape (\Alph*)}, leftmargin=2pc]}{\end{enumerate}}
\newenvironment{enumerate-(A)-r}{\begin{enumerate}[label={\upshape (\Alph*)}, leftmargin=2pc,resume]}{\end{enumerate}}
\newenvironment{enumerate-(i)}{\begin{enumerate}[label={\upshape (\roman*)}, leftmargin=2pc]}{\end{enumerate}}
\newenvironment{enumerate-(i)-r}{\begin{enumerate}[label={\upshape (\roman*)}, leftmargin=2pc,resume]}{\end{enumerate}}
\newenvironment{enumerate-(I)}{\begin{enumerate}[label={\upshape (\Roman*)}, leftmargin=2pc]}{\end{enumerate}}
\newenvironment{enumerate-(I)-r}{\begin{enumerate}[label={\upshape (\Roman*)}, leftmargin=2pc,resume]}{\end{enumerate}}
\newenvironment{enumerate-(1)}{\begin{enumerate}[label={\upshape (\arabic*)}, leftmargin=2pc]}{\end{enumerate}}
\newenvironment{enumerate-(1)-r}{\begin{enumerate}[label={\upshape (\arabic*)}, leftmargin=2pc,resume]}{\end{enumerate}}
\newenvironment{itemizenew}{\begin{itemize}[leftmargin=2pc]}{\end{itemize}}

\newtheorem{theorem}{Theorem}[section]
\newtheorem{lemma}[theorem]{Lemma}
\newtheorem{corollary}[theorem]{Corollary}
\newtheorem{proposition}[theorem]{Proposition}
\newtheorem{question}[theorem]{Question}

\newtheorem{fact}[theorem]{Fact}
\newtheorem{claim}{Claim}[theorem]
\theoremstyle{definition}
\newtheorem{definition}[theorem]{Definition}

\theoremstyle{remark}
\newtheorem{remark}[theorem]{Remark}

\numberwithin{equation}{section}

\newtheorem{innercustomgeneric}{\customgenericname}
\providecommand{\customgenericname}{}
\newcommand{\newcustomtheorem}[2]{%
  \newenvironment{#1}[1]
  {%
   \renewcommand\customgenericname{#2}%
   \renewcommand\theinnercustomgeneric{##1}%
   \innercustomgeneric
  }
  {\endinnercustomgeneric}
}

\newcustomtheorem{customthm}{Theorem}
\newcustomtheorem{customlemma}{Lemma}

\title{Forcing, genericity, and CBERs}

\author[F.~Calderoni]{Filippo Calderoni}

\address{Department of Mathematics, Rutgers University, Hill Center for the Mathematical Sciences, 110 Frelinghuysen Rd., Piscataway, NJ 08854-8019} \email{filippo.calderoni@rutgers.edu}

\author[D.~Sinapova]{Dima Sinapova}
\address{Department of Mathematics, Rutgers University, Hill Center for the Mathematical Sciences, 110 Frelinghuysen Rd., Piscataway, NJ 08854-8019} \email{dima.sinapova@rutgers.edu}


\thanks{The research of the first author was supported by the NSF grant DMS -- 2348819. The research of the second author was supported by the NSF grant DMS -- 2246781. We thank the anonymous referee for their careful reading and for pointing out some inaccuracies.}

\begin{document}

\begin{abstract}
    In this paper we continue the study of equivalence of generics filters started by Smythe in \cite{Smy22}. We fully characterize those forcing posets for which the corresponding equivalence of generics is smooth using the purely topological property of condensation. Next we leverage our characterization to show that there are non-homogeneous forcing for which equivalence of generics is not smooth. Then we prove hyperfiniteness in the case of Prikry forcing and present some additional results addressing the problem of whether generic equivalence for Cohen forcing is hyperfinite.
\end{abstract}

\maketitle

\section{Introduction}

Let \(M\) be a countable transitive model of a sufficient fragment of \(\textsf{ZFC}\), and fix a partial order
\(\mathbb{P}\) in \(M\).
We can define the Polish space $\Gen(M, \mathbb{P})$ of $M$-generic filters for $\mathbb{P}$.  This space has a basis of clopen sets of the form $N_p=\{G\in \Gen(M, \mathbb{P})\mid p\in G\}$ for all $p\in \mathbb{P}\cap M$.
We say that two generic filters \(G,H\in \Gen(M, \mathbb{P})\)
are \markdef{equivalent} (in symbols, \(G \equiv_{M}^\mathbb{P} H\)) if and only if they produce the same generic extension. I.e., we define
\[
G \equiv_{M}^\mathbb{P} H\qquad \iff\qquad M[G]=M[H].
\]

The equivalence relation \(\equiv_{M}^\mathbb{P}\) was analyzed from the point of view of Borel complexity theory by Clemens, Coskey and Dworetzky~\cite{CleCosDwo}, and Smythe~\cite{Smy22}.
By the definability of the forcing relation, this is a countable Borel equivalence relation on the Polish space $\Gen(M, \mathbb{P})$. (See \cite[Lemma~2.6]{Smy22}.)
The case where \(\mathbb{P}\) is the Cohen forcing \(\mathbb{C}\) is particularly compelling, and it is motivated by a long-standing open problem in descriptive set theory. Recall that a Borel equivalence relation is \markdef{hypefinite} if it is the increasing union of a sequence of Borel equivalence relations with finite equivalence classes.
While \cite[Theorem~3.1]{Smy22} proves that \(\equiv_{M}^\mathbb{C}\) is an increasing union of hyperfinite equivalence relations, it is unclear whether \(\equiv_{M}^\mathbb{C}\) is hyperfinite.
This is a particular instance of a central problem in descriptive set theory: the union problem.

The union problem was first posed by Doughety, Jackson, and Kechris~\cite{DouJacKec} in their seminal work about hyperfinite equivalence relations. Some recent striking results show that certain regularity assumptions such as Borel asymptotic dimension and \(u\)-amenability provide a positive solution to the union problem. (See respectively \cite{CJMSTD} and \cite{NarVac}.) However, the union problem remains open and is entangled with another central theme in descriptive set theory: the nuance between hyperfiniteness and measure hyperfiniteness.

On the other hand, Frisch, Shinko and Vidnyanszky~\cite{FriShiVid} recently explored the scenario of the union problem being false and its complexity consequences. 

Back to generic equivalence, people have wondered whether \(\equiv_{M}^\mathbb{C}\) constitutes a counterexample to the union problem. The combinatorial nature of forcing posets suggests that one could develop new techniques to study hyperfiniteness in the more restrictive context of generic equivalence. (For a discussion about this, we refer the reader to~\cite[Page~82]{Kec24}). 

In addition, Smythe~\cite[Question~1]{Smy22} asks to determine the Borel complexity of equivalence of generics for other forcings. In this paper we continue on this theme. We primarily focus on the problem of determining smoothness and hyperfiniteness for \(\equiv^\mathbb{P}_M\).

In Section~\ref{sec : smoothness} we provide a new characterization of smoothness for \(\equiv_{M}^\mathbb{P}\). In Sections~\ref{sec : Prikry}  we prove that \(\equiv_{M}^\mathbb{P}\) is hyperfinite in the case of Prikry forcing. In Section~\ref{sec : Cohen} we prove some additional results about generic equivalence, and we discuss some equivalence relations closely related to \(\equiv^\mathbb{C}_M\). We believe that these would be useful to address the problem of finding the Borel complexity of \(\equiv^\mathbb{C}_M\) and other forcings in the future.

\subsection*{Postscript} As we completed the latest revision of this paper in November~2026, it was announced that \(\equiv^\mathbb{C}_M\) is hyperfinite in the case of Cohen forcing. For a proof see the preprint~\cite{Smy26}.

\section{Preliminaries}
\label{sec:prelim}

\subsection{Forcing}

  Recall that $\mathbb{P}$ is \markdef{weakly homogeneous} if for all pairs of conditions $q,r$, there is an automorphism $e\colon\mathbb{P}\rightarrow\mathbb{P}$, such that $e(q)$ is compatible with $r$. Note that this implies that there are generic filters $G,H$, such that $q\in G, r\in H$ and $V[H]=V[G]$. Also, recall that if, on the other hand $G$ and $H$ are two generic filters for $\mathbb{P}$ with  $V[H]=V[G]$, and $\overline{\mathbb{P}}$ is the Boolean completion of the poset, then there is an  automorphism $e\colon\overline{\mathbb{P}}\rightarrow \overline{\mathbb{P}}$, such that $e"G=H$.

Weak homogeneity can be used to detect non-smoothness in the following precise sense. Recall that a poset is \markdef{atomless} if and only if every condition has incompatible extensions. The following result can be found in~\cite[Theorem~2.12]{Smy22}.

\begin{theorem}[Smythe~2022]
\label{thm : Smythe not smooth}
If $\mathbb{P}$ is atomless and weakly homogeneous, then $\equiv_{M}^\mathbb{P}$ is not smooth.
\end{theorem}

In fact, Smythe proved that $\mathbb{P}$ is weakly homogeneous if and only if the action of $\Aut^M(\mathbb{P})$ on $\Gen(M, \mathbb{P})$ is generically ergodic. (See~\cite[Lemma~2.11]{Smy22}.) In Proposition~\ref{prop : weakly homogeneous}  below, we show that the sufficient condition in Theorem~\ref{thm : Smythe not smooth} is not necessary.

For a forcing poset \(\mathbb{P}\), denote by \(\mathbb{\overline P}\) its Boolean completion as computed in \(M\). We will make use of the following fact due to Smythe~\cite[Lemma~2.7]{Smy22}:

\begin{fact}
\label{fact: generic equivalence isomorphic to completion}
 The equivalence relations \(\equiv^{\mathbb{P}}_M\) and \(\equiv^{\overline{\mathbb{P}}}_M\) are Borel isomorphic. In particular, \(\equiv^{\mathbb{P}}_M\) is smooth (resp., hyperfinite) if and only if \(\equiv^{\overline{ \mathbb{P}}}_M\) is smooth (resp., hyperfinite).
\end{fact}

We will say that $\mathbb{P}$ is \markdef{densely weakly homogeneous} if for densely many $p$, the poset $\mathbb{P}\downarrow p = \{q\in \mathbb{P}\mid q\leq p\}$ is weakly homogeneous. For an example of a densely weakly homogeneous poset that is not weakly homogeneous, take a lottery sum of non-isomorphic forcings, $\mathbb{P}_i$, such that for each $ p\in\mathbb{P}_i$, the poset  \(\mathbb{P}_i\downarrow p\) is weakly homogeneous.  Recall that the \markdef{lottery sum} of a collection \(\{\mathbb{P}_i:i\in I\}\) of forcing posets is the
poset \(\mathbb{P}=\bigoplus_{i\in I} \mathbb{P}_i = \{\langle  i, p\rangle : i\in I \text{ and } p \in \mathbb{P}_i\}\cup\{\mathbb{1}\}\), ordered with \(\mathbb{1}\) weaker than everything else and
\(\langle i, p\rangle \leq \langle j, q\rangle\) when \(i = j\) and \(p\leq_{\mathbb{P}_i} q\).  
Then for all $\langle i, p\rangle$, the forcing below 
 $\langle i, p\rangle$, the poset $\mathbb{P}\downarrow {\langle i, p\rangle}$, is isomorphic to \(\mathbb{P}_i\downarrow p\),  which by assumption is weakly homogeneous, and so the full forcing is densely weakly homogeneous. On the other hand, since the $\mathbb{P}_i$'s are non-isomorphic, the full forcing is not weakly homogeneous.

\subsection{Countable Borel equivalence relations}

An equivalence relation \(E\) is countable if every \(E\)-class is countable. It is finite if every \(E\)-class is finite.

Let \(X = 2^\mathbb{N}\). Then the \markdef{eventual equality relation} is
\[
x\mathbin{E}_0y\quad \iff\quad \exists m \forall n \geq m(x_n = y_n).
\]
It is immediate that \(E_0\) is hyperfinite because \(E_0 = \bigcup_{i\in\mathbb{N}} F_i\), where \(x\mathbin{F_i} y \iff x_n= y_n\) for all \(n\geq i\). Clearly each \(F_i\) class has cardinality \(2^i\).

Let \(\Gamma\) be a countable group with the discrete topology.
If \(X\) is a Polish space and \(a\colon \Gamma\times X \to X\) is a continuous action of \(\Gamma\) on \(X\) we say that \(X\) is a \markdef{Polish \(\Gamma\)-space}. In which case, we denote by \(\mathcal{R}(\Gamma\curvearrowright X)\) the corresponding orbit equivalence relation.
Given a Polish \(\Gamma\)-space $X$, a point $x\in X$ is \markdef{condensed} if \(x\) is an accumulation point of its orbit \([x]\). That is,  for all open neighborhoods $U$ of $x$, $|U\cap[x]|\geq 2$. In the sequel, we make use of the following well-known fact. For a proof see e.g.~\cite[Proposition~2.2]{CalCla23'}.

\begin{fact}
\label{fact:condensed point}
Let \(X\) be a Polish \(\Gamma\)-space.  The orbit equivalence relation \(\mathcal{R}(\Gamma\curvearrowright X)\) is not smooth if and only if there is a condensed point.
\end{fact}

\section{Characterization of smoothness}
\label{sec : smoothness}

The main goal in this section is to present a characterization of non-smoothness for \(\equiv_{M}^\mathbb{P}\). (Theorem~\ref{lemma : dagger} below.)
First, we spell out condensation in the case when $X$ is $\Gen(M, \mathbb{P})$ and \(\Gamma=\Aut^M(\mathbb{P})\).
Recall that \(N_p=\{G\in \Gen(M, \mathbb{P})\mid p\in G\}\). It is crucial for us that the set \(\{N_p \mid p\in \mathbb{P}\cap M\}\) is a basis for the space \(\Gen(M,\mathbb{P})\). (For the details, see~\cite[Lemma~2.3]{Smy22}.) Therefore, we can refer to the basic (cl)open sets \(N_p\), rather than referring to all open subsets of the space.

Below we consider the action of \(\Gamma \curvearrowright \Gen(M,\mathbb{P})\). We recall the following facts from Smythe~\cite{Smy22}:

\begin{fact}
\label{generic fact}
    \begin{enumerate-(i)}
        \item 
        \label{generic fact:i}
        The orbit equivalence relation \(\mathcal{R}(\Gamma\curvearrowright \Gen(M,\mathbb{P}))\) is a subequivalence relation of \(\equiv^{\mathbb{P}}_M\).
        \item
        \label{generic fact:ii}
        If \(\mathbb{P}\) is a complete Boolean algebra, then \(\mathcal{R}(\Gamma\curvearrowright \Gen(M,\mathbb{P}))\) and \({\equiv^\mathbb{P}_M}\) coincide.
        \end{enumerate-(i)}
\end{fact}

We also note that we do not in general have equality in Fact~\ref{generic fact}\ref{generic fact:i}. --- As mentioned by Smythe following~\cite[Theorem~2.12]{Smy22}, every poset is forcing equivalent to one that has a trivial automorphism group.

Inspired by the notion of condensed points discussed in Section~2, we introduce the following definition.

\begin{definition}
\label{def : condensed filter}
An \(M\)-generic filter $G\in \Gen(M, \mathbb{P})$ is \markdef{condensed} if for all $N_p$ with $G\in N_p$, there is $H\neq G$, such that $M[H]=M[G]$ and $H\in N_p$. 
\end{definition}

\begin{remark}
\label{ramark:condensed filter completion}
Note that whenever \(\mathbb{P}\) is a complete Boolean algebra, \(G\) is condensed according to Definition~\ref{def : condensed filter} if and only if \(G\) is a condensed point for the action \(\Gamma\curvearrowright \Gen(M,\mathbb{P})\) in the sense of Section~\ref{sec:prelim}.
\end{remark}

 In general, Fact~\ref{generic fact}\ref{generic fact:i} implies that if \(G\) is condensed, then \(G\) is an accumulation point of its \(\equiv^\mathbb{P}_M\)-class.

\begin{definition} Let \(\mathbb{P}\) be a forcing poset in \(M\). For any \(p\in \mathbb{P}\) define the statement \eqref{dwh} as follows:
\begin{multline}\tag{$\dagger_p$}\label{dwh}
\text{For all $p'\leq p$, there are incompatible $q,r\leq p'$,}\\
\text{such that there exists distinct generic filters $G,H$,}\\
\text{such that $q\in G, r\in H$ and $M[H]=M[G]$.}
\end{multline}
\end{definition}

Roughly speaking, whenever \(\mathbb{P}\) is a complete Boolean algebra, condition~\eqref{dwh} says that we can find instances of homogeneity densely often below $p$. In case \(\mathbb{P}\) is atomless, densely weakly homogeneous implies  $(\dagger)_p$ for all $p\in \mathbb{P}$.
We will show that this is strictly weaker than being densely weakly homogeneous.

\begin{proposition}
\label{prop:dagger completion}
    As \(\mathbb{P}\) and \(\mathbb{\overline{P}}\) are forcing equivalent, \eqref{dwh} holds for some \(p\in \mathbb{P}\) if and only if \((\dagger_{p'})\) holds for some \(p'\in \overline{\mathbb{P}}\).
\end{proposition}

\begin{proof}
Recall that \(\mathbb{P}\) densely embeds into  \(\overline{\mathbb{P}}\).
For the non-trivial direction, if  \((\dagger_{p'})\) holds for some \(p'\in \overline{\mathbb{P}}\), let $p\in\mathbb{P}$ be stronger than $p'$. Then, there are incompatible $q',r'\in \overline{\mathbb{P}}$, stronger than $p$, and generic filters $G',H'\subset \overline{\mathbb{P}}$ containing $q',r'$, respectively, such that $M[G']=M[H']$. By the forcing equivalence, there is a generic filter $G\subset \mathbb{P}$, such that $M[G]=M[G']$ with $q\leq q'$, $q\in\mathbb{P}$, and similarly a filter $H\subset\mathbb{P}$ with  $M[H]=M[H']$ $r\leq r'$, $r\in H$. Since this argument can be carried out below any condition stronger than $p$, we get \eqref{dwh}.
\end{proof}

\begin{theorem}
\label{lemma : dagger}
Let \(\mathbb{P}\) be a poset in \(M\). The equivalence relation $\equiv_{M}^\mathbb{P}$ is not smooth if and only if \eqref{dwh} holds for some $p\in \mathbb{P}$.

\end{theorem}
\begin{proof}
For the forward direction,  if $\equiv_{M}^\mathbb{P}$ is not smooth, then the equivalence relation $\equiv_{M}^{\overline{\mathbb{P}}}$ is not smooth either.
By Fact~\ref{fact:condensed point} and Remark~\ref{ramark:condensed filter completion}, we conclude that
there is a condensed filter \(G\in\Gen(M,\overline{\mathbb{P}})\). It follows that \((\dagger_{p'})\) holds for all $p'\in G$. (Use Definition~\ref{def : condensed filter}.) Therefore, \eqref{dwh} holds for some \(p\in \mathbb{P}\) by Proposition~\ref{prop:dagger completion}.

For the right to left direction,
assume \eqref{dwh} for some \(p\in\mathbb{P}\).
We may assume that \(\mathbb{P} = \overline{\mathbb{P}}\) without loss of generality because of Fact~\ref{fact: generic equivalence isomorphic to completion} and Proposition~\ref{prop:dagger completion}.
Let $G\subseteq \mathbb{P}$ be a generic filter with $p\in G$. We want to show that $G$ is condensed.   Namely,  for all \(p'\in G\) there is \(H\neq G\) with  and \(M[G] = M[H]\) and \(H \ni p'\).

Fix \(p'\in G\). By strenghtening if necessary, we may assume $p'\leq p$. First we prove the following density lemma.
\begin{proof}[Claim]
The set \(\mathcal{D} = \{ q : \exists e \in \Aut^M(\mathbb{P})\text{ such that }e(q) \perp q, e(q)\leq p' \}\) is dense in \(\mathbb{P}\) below  \(p'\).
\renewcommand{\qed}{}
\end{proof}
\begin{proof}
Given \(q\leq p' \leq p\) by \eqref{dwh} we can find incompatible conditions \(q', r\leq q\) such that
\[
q'\in H_1 \quad\quad r\in H_2 \quad\quad M[H_1] = M[H_2].
\]

Then there is \(e\in \Aut^M(\overline{\mathbb{P}})\) such that $e"H_1=H_2$. So, \(e(q')\) and \(r\) are compatible. Then let \(r^* \leq e(q'), r\) and denote by \(q^* = e^{-1}(r^*)\). It follows that \(q^*\leq q'\) and also that \(q^* \) and \(e(q^*)=r^*\) are incompatible. Then $q^*\in \mathcal{D}$.
\end{proof}

Choose \(q\in\mathcal{D}\cap G\) so that \(q\leq p'\).  Let \(e\in \Aut^M(\mathbb{P})\) witness that \(q\in\mathcal{D}\). By letting \(H = e'' G\) we get \(M[G] = M[H]\) and since \(q\) and \(e(q)\) are incompatible it follows that \(G,H\) are different. And since both \(q\) and \(e(q)\) are below $p'$, both $G,H$ contain $p'$, as desired.
\end{proof}

Using condensation we show the following:
\begin{proposition}
\label{prop : weakly homogeneous}
There is a poset $\mathbb{P}_\ell$, whose completion is not weakly homogeneous below any condition,  but \(\equiv_{M}^{\mathbb{P}_\ell}\) is not smooth.
\end{proposition}

Towards the proof of Proposition~\ref{prop : weakly homogeneous}, we will define \(\mathbb{P}_\ell\) whose completion is not densely weakly homogeneous,  but it satisfies \eqref{dwh} for all \(p\in\mathbb{P}_\ell\).
The idea is to take lottery sums of non-isomorphic weakly homogeneous forcings in a tree like fashion. 
 Let $\mathbb{P}_0$ be a weakly homogeneous forcing poset. For ease of exposition we take a constant sequence of forcing \(\{\mathbb{P}^i_0: i\in \mathbb{N}\}\) with \(\mathbb{P}^i_0 = \mathbb{P}_0\). Let \(\{\mathbb{P}_1^i:i\in\mathbb{N}\}\) be a sequence of forcings whose completion is non-weakly homogeneous and  such that no \(\mathbb{P}^i_1\) projects\footnote{Here \(\mathbb{P}\) projects onto \(\mathbb{Q}\) means that whenever \(G\) is generic for \(\mathbb{P}\), one can define a generic for \(\mathbb{Q}\) in \(V[G]\).} into \(\mathbb{P}^j_1\) for \(i\neq j\).
For nonempty \(\sigma\in 2^{<\omega}\), let \(\mathbb{P}_\sigma\) be \(\prod_{n<|\sigma|}\mathbb{P}^n_{\sigma(n)}\), so that $\mathbb{P}_{01}=\mathbb{P}_0\times\mathbb{P}^1_1, \mathbb{P}_{11}=\mathbb{P}^0_1\times\mathbb{P}^1_1,\mathbb{P}_{110}=\mathbb{P}^0_1\times\mathbb{P}^1_1\times\mathbb{P}_0$, etc.  Finally, we define \(\mathbb{P}_\ell\) as follows:
Conditions in $\mathbb{P}_\ell$ are of the form $(\sigma, p)$ for some $\sigma\in 2^{<\omega}$ and \(p\in \mathbb{P}_\sigma\), and

\[(\sigma, p)\leq (\tau,q)\qquad\iff\qquad\sigma \supseteq \tau\text{ and } {p\restriction \tau} \leq_{\mathbb{P}_\tau} q.\]

\begin{lemma}
The completion of \(\mathbb{P}_\ell\) is not densely weakly homogeneous,  but \eqref{dwh} holds for all $p$.
Therefore, \(\equiv_{M}^{\mathbb{P}_\ell}\) is not smooth.
\end{lemma}
\begin{proof}
Since \(\mathbb{\overline{P}}^i_1\) is not weakly homogeneous for each $i$, we can find conditions \(q_1,r_1\in \mathbb{P}^i_1\) such that for all \(\mathbb{P}^i_1\)-generics \(G_1\ni q_1\) and \(H_1\ni r_1\) we have \(M[G_1]\neq M[H_1]\).
It follows that the poset $\overline{\mathbb{P}}_\ell$ is not weakly homogeneous because any condition $(\sigma,p)$ can be extended to incompatible conditions $({\sigma\cat 1}, p\cat q_1)$ and \(({\sigma\cat 1}, p\cat r_1)
\). Then if \(G,H\) are \(\mathbb{P}_\ell\)-generic containing $({\sigma\cat 1}, p\cat q_1)$ and \(({\sigma\cat 1}, p\cat r_1)\) respectively, and \(M[G]=M[H]\), then  the projections of \(G,H\) on the \(|\sigma|\) coordinate must produce the same model.

On the other hand,  any condition $(\sigma,p)$ can be extended to conditions  $({\sigma\cat 0}, q)$, $({\sigma\cat 0},r)$ such that \(q,r\in\mathbb{P}_{\sigma\cat 0}\) and  $q\upharpoonright |\sigma|=r\upharpoonright |\sigma|=p$. Using homogeneity of $\mathbb{P}_0$ we can find generics containing these conditions that produce the same extension.  Therefore,  we conclude that \(\mathcal{R}(\Aut^M(\mathbb{P}_\ell)\curvearrowright \Gen(\mathbb{P}_\ell, M))\) is not smooth by Fact~\ref{fact:condensed point},  and neither \(\equiv_{M}^{\mathbb{P}_\ell}\) is.
\end{proof}

\section{Prikry forcing is hyperfinite}
\label{sec : Prikry}

Let $\kappa$ be a measurable cardinal and fix a normal measure $U$ on $\kappa$ (i.e. $U$ is a $\kappa$-complete non-principal normal ultrafilter on $\kappa$). In this section $\mathbb{P}$ will be the Prikry forcing at $\kappa$ with respect to $U$. I.e. conditions are of the form $(s, A)$, where $s$ is a finite increasing sequence of points in $\kappa$ and $A\in U$. We will refer to sets in $U$ as ``measure one".  We say that \((s',A')\leq (s,A)\) if and only if \(s'\) is an end-extension of \(s\),  \(s'\setminus s \subseteq A\) and \(A'\subseteq A\). For $(s, A)\in\mathbb{P}$, we call $s$ the stem of the condition, and $|s|$ is the length. We say $q$ is a direct extension of $p$, denoted $q\leq^* p$ if $q\leq p$ and they have the same stem.

Let $M$ be a countable model satisfying enough of \textsf{ZFC}, containing $\kappa$, \(U\) and $\mathbb{P}$. In particular, \(M\models ``\kappa\text{ is measurable}''\).

A generic filter for $\mathbb{P}$ induces a generic Prikry sequence $\vec{\alpha}=(\alpha_n: n<\omega)$ cofinal in $\kappa$, and it is actually defined from it as follows: given such a sequence $\vec{\alpha}$, define $G$ as the set of conditions $( s, A)$, such that for some $n$, $s = ( \alpha_0, \dotsc, \alpha_{n-1})$ and for all $i\geq n$,  $\alpha_i\in A$.  We say that $\vec{\alpha}$ is a \markdef{generic sequence} if the filter induced from it is generic. We will be identifying generic filters with the corresponding generic Prikry sequences.

By identifying $M$ with its transitive collapse, we can assume that \(M\) is transitive without any loss of generality.

\begin{theorem}
\label{thm : Prikry}
$\equiv_{M}^\mathbb{P}$ is hyperfinite. In fact, \({\equiv_{M}^\mathbb{P}}\sim_B E_0\).
\end{theorem}

We will use the following facts about Prikry forcing.
\begin{fact}
\label{fact : Prikry}
\begin{enumerate-(1)}
\item
\label{fact : Prikry(1)}
If $\vec{\alpha}$ and $\vec{\beta}$ are both $M$-generic sequences for \(\mathbb{P}\), then $M[\vec{\alpha}]=M[\vec{\beta}]$ if and only if $\vec{\alpha}$ and $\vec{\beta}$ coincide on a tail. I.e., there are \(m,n\in \omega\) such that \(\alpha(m+k) = \beta(n+k)\) for all \(k\in\omega\).
\item
\label{fact : Prikry(2)}
(The Prikry lemma). For any condition $p$ and a dense open set $D$, there is a direct extension $q\leq^*p$ and $k<\omega$, such that every $r\leq q$ of length $k$ is in $D$.
\item
\label{fact : Prikry(3)}
Prikry forcing is weakly homogeneous. In particular, for any $M$-generic $G$ for $\mathbb{P}$ and any condition $p\in M$, there is a $M$-generic $H$ for $\mathbb{P}$, such that $p\in H$ and $M[G]=M[H]$.

\end{enumerate-(1)}
\end{fact}

Facts~\ref{fact : Prikry(2)}--\ref{fact : Prikry(3)} can be found in Gitik~\cite[Section~1]{Git10}.
Fact~\ref{fact : Prikry}\ref{fact : Prikry(1)} implicitly follows from the results of Gitik, Kanovei and Koepke~\cite{GitKanKoe}.
Since this work is unpublished, we provide a proof for completeness.

\begin{proof}[Proof of Fact~\ref{fact : Prikry}\ref{fact : Prikry(1)}]
One direction is clear. For the other direction, suppose for contradiction that $\vec{\alpha}$ and $\vec{\beta}$ are two generic Prikry sequences over $M$, such that $M[\vec{\alpha}]=M[\vec{\beta}]$, but the symmetric difference of these sequences is infinite. It follows that there is a subsequence of $\vec{\beta}$, which is disjoint from $\vec{\alpha}$. Let $\vec{\gamma}$ enumerate this subsequence. Note that $\vec{\gamma}\in M[\vec{\alpha}]$.

Let $p$ in the generic filter induced by $\vec{\alpha}$, force that $\dot{\vec{\gamma}}$ is disjoint from the generic Prikry sequence. 

For each $n$, let $D_n$ be the dense set of conditions deciding a value for $\dot{\gamma}_n$. By the Prikry lemma, there is a direct extension $p'\leq^*p$, and $k<\omega$, such that every $k$-step extension of $p'$ is in $D_n$. I.e., for every $r\leq p'$ with length $\lh(r)=\lh(p)+k$, there is $\delta<\kappa$, such that $r\Vdash \delta=\dot{\gamma}_n$.

\begin{claim}
Fix $n$, and suppose that $k$ is minimal such that the above holds for some $p'\leq^*p$ and $D_n$, and denote $p' = (h, A)$. Then there is a measure one set $B\subset A$, such that for all increasing $s$ in $B^k$, if $\delta_s$ is the unique \(\delta\) such that $(h\cat s, B\setminus \max(s)+1)\Vdash\delta=\dot{\gamma}_n$, we have that $\max(s)<\delta_s$.
\end{claim}
\begin{proof}
Suppose otherwise. Since $p'$ forces that $\dot{\gamma}_n$ is not on the Prikry sequence, each $\delta_s\notin s$.  Let $\phi\colon[A]^k\rightarrow \{0,1\}$ be the function defined as follows
\[
\phi(s) = 
\begin{cases}
0 & \text{if $\max(s)<\delta_s$}\\
1 & \text{if $\max(s)>\delta_s$}.
\end{cases}
\]
By measurability, $\phi$ is homogeneous on a measure one set, namely, there is $B\subset A$, $B\in U$, such that for all $k$, $\phi\upharpoonright [B]^k$ is constant. So, for all $k$-length stems $s$ taken from $B$, $\delta_s<\max(s)$. Then for each $t\in [B]^{k-1}$, the function $g_t\colon \alpha\mapsto \delta_{t\cat\alpha}$ is regressive. By Fodor's Theorem (e.g., see~\cite[Theorem~8.7]{Jec}), there is a measure one set $B_t\subset B$ and $\delta_t<\kappa$, such that for all $s$ of the form $t\cat\alpha$ with $\alpha\in B_t$, $\delta_s=\delta_t$. Let $B'=\triangle B_t = \{\alpha\mid \alpha\in\bigcap_{\max(t)<\alpha} B_t\}$. Then if $r\leq (h, B')$ is any condition with length $\lh(p)+k-1$; denote $r = (h\cat t, C)$. We may assume that $\max(t)<\min(C)$.  Then $C\subset B_t$, and so $r\Vdash \dot{\gamma}_n=\delta_t$. Contradiction with the minimality of $k$.
\end{proof}

Using the above claim and inductively applying the Prikry lemma, for each $n$, we can find a condition $p' = (h, A)\leq^*p$, and a sequence $( k_n: n<\omega)$, such that:
\begin{enumerate}
\item
for each $n$, every $k_n$-step extension of $p'$ decides the value of $\dot{\gamma}_n$,
\item
for each stem $s$ of points in $A$ of length $k_n$, setting $\delta_s$ to be such that $(h\cat s, A\setminus \max(s)+1)\Vdash\dot{\gamma}_n=\delta_s$, we have that $\delta_s>\max(s)$.

\end{enumerate}

Now, let $G$ be the generic filter induced by $\vec{\alpha}$.
Since we can do this densely often, we can find such a condition in  $G$. Now go back to $M[G]$. Let $n$ be such that $\gamma_n\in A$. Let $q\leq p$ be the (unique) weakest $k_n$-step extension of $p$ in $G$. Namely, we let $q = (h\cat s, A\setminus \max(s)+1)$, where \(h\cat s\) is an initial segment of \(\vec \alpha\) of length \(|h|+k_n\). Then we have that $q\Vdash\dot{\gamma}_n=\delta_s$. Since $q\in G$, then $\delta_s=\gamma_n\in A$. And since $\max(s)<\delta_s$, we can further extend $q$ to $r = (h\cat {s\cup\{\delta_s\}}, A\setminus \delta_s+1)$. But then $r$ forces that $\dot{\gamma}_n$ is on the Prikry sequence. Contradiction with our initial assumption.
\end{proof}

\begin{proof}[Proof of Theorem~\ref{thm : Prikry}]
Since \(M\) is countable, we identify \((\kappa\cap M)^\omega\) with the Baire space. Roughly, speaking we encode \(\equiv^\mathbb{G}_M\) into a hyperfinite equivalence relation on the standard Borel space of \(\mathcal{G}\) of generic Prikry sequences in \((\kappa\cap M)^\omega\). Precisely, since
there is a one-to-one Borel function $f\colon\Gen(M,\mathbb{P})\rightarrow (\kappa\cap M)^\omega$ associating to any filter the corresponding generic Prikry sequence, let \(\mathcal{G} = f'' \Gen(M,\mathbb{P})\) which is injective image of a Borel set, therefore Borel.

Working in $V$, the set $\kappa\cap M$ is countable. By Fact~\ref{fact : Prikry}
\ref{fact : Prikry(1)}, the equivalence relation \(\equiv^\mathbb{P}_M\) coincide with the restriction of the tail equivalence relation \(E_t(\kappa\cap M)\) on the countable set \(\kappa\cap M\) to the space of generic Prikry sequences \(\mathcal{G}\). And \( E_t(\kappa\cap M)\) is hyperfinite by ~\cite[Corollary~8.2]{DouJacKec}. It follows that $\equiv_{M}^\mathbb{P}$ is hyperfinite, and in particular $\equiv_{M}^\mathbb{P}\sim_B E_0$. 

\end{proof}

\section{Towards hyperfiniteness via mutual genericity}

\label{sec : Cohen}

Let $M$ be a countable model of a large enough finite fragment of \(\mathsf{ZFC}\). We do not assume $M$ to be transitive, but it is wellfounded and could be identitied with its transitive collapse.  Let $\mathbb{P}$ be a complete Boolean algebra in \(M\); by elementarity $2^{\mathbb{P}}\in  M$. Note that while $\mathbb{P}$ itself is not countable, $M\cap \mathbb{P}$ is. Let $X=\Gen(M,\mathbb{P})$ be the set of $M$-generic filters for $\mathbb{P}$. We have $X\notin M$, but $X\subset 2^{\mathbb{P}\cap M}$.

Throughout this section, we analyze the equivalence relation $\equiv_{M}^\mathbb{P}$ on $X$ and discuss the interplay between invariant uniformization and hyper-hyperfiniteness. (Invariant uniformization of equivalence relations is also discussed in a recent paper by Kechris and Wolman~\cite{KecWol}.) Moreover, we show that a weak form of invariant uniformization implies hyper-hyperfiniteness for \(\equiv^\mathbb{P}_M\). (See Proposition~\ref{prop : weak uniformization}.)

We simplify our notation by letting \(E \) be $\equiv_{M}^\mathbb{P}$.
By Fact~\ref{generic fact}\ref{generic fact:ii} we have that $E$ is generated by the countable group $\Aut^M(\mathbb{P})$. Even more precisely, if $xEy$, then there is a an involution $\gamma\Aut^M(\mathbb{P})$, such that $x=\gamma"y$. 
Let $\{\gamma_n\mid n<\omega\}$ enumerate the involutions in $\Aut^M(\mathbb{P})\).   For each $n$, let $\Gamma_n$ be the group generated by  $\{\gamma_i\mid i\leq n\}$. Note that while the full enumeration of the automorphisms is not in $M$, for each $n$, \(\Gamma_n\) is in $M$.
Let $E_{\Gamma_n}$ be the induced orbit equivalence relation on $X$. 
Then $E_{\Gamma_n}\subseteq E_{\Gamma_{n+1}}$ and $E=\bigcup_n E_{\Gamma_n}$. 

While it is not quite true that \(E_{\Gamma_n}\) is in \(M\), since it is an equivalence relation on a Polish space that is not in $M$, we have that \(E_{\Gamma_n}\) is generated by a set of automorphisms in $M$. Moreover, we can assume those automorphism are involutions:

\begin{lemma}
For each $n$, there are involutions
$g_k\in \Aut^M(\mathbb{P})$ and an enumeration
$\{g_k\mid k<\omega\}\in M$, such that for all $M$-generics $x,y$, we have $x\mathbin{E}_{\Gamma_n}y$ if and only if there is $k\in \omega$, such that $x=g_k"y$.
\end{lemma}
\begin{proof}
The proof of \cite[Theorem~3.5.1]{Gri75} gives that for every $\delta\in \Aut^M(\mathbb{P})$, there is an involution $\gamma\in \Aut^M(\mathbb{P})$, such that $1_{\mathbb{P}}\Vdash\gamma"\dot{G}=\delta"\dot{G}$, where
$\dot{G}$ is the canonical name for the generic filter.

Working in \(M\), enumerate \(\Gamma_n = \{\delta_k\mid k< \omega\}\). For each \(k\) let \(g_k\) be an involution in \(\Aut^M(\mathbb{P})\) such that \(1_{\mathbb{P}}\Vdash g_k"\dot{G}=\delta_k"\dot{G}\). We claim that \(\{g_k\mid k<\omega\}\) is as desired.

This set is in $M$. Also, for each $k$, if $g_k"x=y$, then  $1_{\mathbb{P}}\Vdash g_k"\dot{G}=\delta_k"\dot{G}$, and so $\delta_k "x=g_k"x =y$. It follows that $x\mathbin{E}_{\Gamma_n}y$.

For the other direction, suppose that $x\mathbin{E}_{\Gamma_n}y$, and let \(k\) be such that $\delta_k"x=y$. Then, since $1_{\mathbb{P}}\Vdash\gamma_k"\dot{G}=\delta_k"\dot{G}$, it follows that $g_k"x=y$.
\end{proof}

\begin{lemma}
\label{lemma : Borel order}
There is a Borel linear ordering $<$ on $2^{\mathbb{P}\cap M}$, such that for any two generics $x,y$ with $xEy$, we have that $x<y$ is definable in $M[x]$.
\end{lemma}
\begin{proof}
Enumerate $\mathbb{P}=\langle p_i\mid i<\tau\rangle$, for some cardinal $\tau$, and by elementarity assume that this enumeration is in $M$. Let $J=\tau\cap M$. Then $J$ is countable and $\mathbb{P}\cap M=\{p_i\mid i\in J\}$.

Let $<$ be the Borel order on  $2^{\mathbb{P}\cap M}$, given by: \[x< y \iff \exists j\in J ( p_j\notin x\text{ and } p_j\in y\text{ and }\forall i\in J\cap j(p_i\in x \iff p_i\in y)).\] In other words, this is isomorphic to the lexicographic order on $2^\delta$, where \(\delta\) is the order type of $(J, <)$.

Then for $x,y$ as in the statement of the lemma, we have that $x<y$ if and only if 
\(M[x]\models
\exists j<\tau (p_j\in y\setminus x
\text{ and }\forall i<j (p_i\in x\iff p_i\in y))\).
\end{proof}

Let $\{I_n\mid n<\omega\}\in M$ be an increasing family of infinite subsets of $\omega$, such that $\bigcup_n I_n=\omega$ and $I_{n+1}\setminus I_n$ is infinite.

Using the above claim for each $n$, find involutions $\{g_k\mid k<\omega\}$, each $g_k$ in $\Aut^M(\mathbb{P})$  (possibly with repetitions), such that for each $n$, $\{g_k\mid k\in I_n\}\in M$ and $x\mathbin{E}_{\Gamma_n}y $ iff for some $k\in I_n,g_k"x=y$. Then, also $x\mathbin{E}y$ iff for some $k<\omega, g_k"x=y$.

Next we define finite Borel equivalence relations \(E_s\) on $X$, for \(s\in \omega^{<\omega}\), so that if $s\subseteq s'$, then $E_s\subseteq E_{s'}$ and each $E_s\subset E$ with a cascade argument as in \cite[Theorem~12.1]{KecMil}.

More precisely, we first fix a Borel linear ordering $<$ on $2^{\mathbb{P}\cap M}$, as in Lemma~\ref{lemma : Borel order}. Then, for each $n<\omega$, let $f_k(x) = \min_{<}\{x, g_k"x\}$; this is a Borel function. For $s=\langle s_0, \dotsc, s_{l-1}\rangle\in\omega^{<\omega}$,  set
\[x \mathbin{E_s} z\qquad\iff\qquad f_{s_{l-1}}\circ \dotsb \circ f_{s_1} \circ f_{s_0}(x)=f_{s_{l-1}}\circ \dotsb \circ f_{s_1} \circ f_{s_0}(z).\]
For $y\in\omega^\omega$, let 
\begin{equation}
\label{eq:E_y}
E_y = \bigcup_n E_{y\upharpoonright n}.
\end{equation}

Next, fix $n$ and let $z\mathbin{E_{\Gamma_n}} x$ and $s\in I_n^{<\omega}$. Let
\[x' = f_{s_{l-1}}\circ \dotsb \circ f_{s_1} \circ f_{s_0}(x)\qquad\text{ and }\qquad z' = f_{s_{l-1}}\circ \dotsb \circ f_{s_1} \circ f_{s_0}(z)\] where $l=|s|$. Then $x' \mathbin{E_{\Gamma_n}} z'$, and so  $g_k(x')=z'$, for some $k\in I_n$.  Let $s'=s\cat k$. Then we have that $z\mathbin{E_{s'}}x$. Take say the mimimun such $s'$ and denote it by $s^n_{x,z}$. Note that for a fixed pair \((x,z)\) the function  $s\mapsto s^n_{x,z}$ is in \(M[x]\) (and that 
 \(M[x]=M[z]\)), even though for any \(s\in \omega^{<\omega}\), the equivalence relation \(E_s\) is not in \(M\). In particular, the set $D^n_{x,z}=\{s^n_{x,z}\mid s\in \omega^{<\omega}\}\in M[x]$.

Now, for any $x,z$ such that $x\mathbin{E} z$, and $s\in \omega^{<\omega}$, let $s'_{x,z}=s^n_{x,z}$, where $n$ is the least such that $z\mathbin{E_{\Gamma_n}} x$ and $s\subset I_n$. 

For $x,z$, such that $x\mathbin{E}z$, let $D_{x,z}=\{s'_{x,z}\mid s\in \omega^{<\omega}\}$. This is a dense subset of the poset $(\omega^{<\omega}, \supseteq)$ (which can be identified with $\Add(\omega, 1)$). And define: $$C_{x,z}=\bigcup_{s\in \omega^{<\omega}}\mathcal{N}_{s'_{x,z}}.$$
Here we use the notation $\mathcal{N}_s\coloneqq\{y\in\omega^\omega\mid s\sqsubset y\}$. Clearly, $C_{x,z}$ is dense open. For each $x$, let $C_x = \bigcap_{z,w\in[x]_E}C_{w,z}$. Then $C_x$ is comeager in $\omega^\omega$. Also, for each $n$ and $x,z$, such that $xE_{\Gamma_n}z$, define $$C^n_{x,z}=\bigcup_{s\in I_n^{<\omega}}\mathcal{N}_{s^n_{x,z}}.$$
Furthermore, for each $x$, let $C^n_x = \bigcap_{z,w\in[x]_{E_{\Gamma_n}}}C^n_{w,z}$. Note that $C^n_x$ is comeager in $I_n^{\omega}$.

We will say that $y\in\omega^\omega$ is \textbf{\(N\)-generic} for some model \(N\), if it is generic for the poset $(\omega^{<\omega},\supseteq)$ over \(N\).
Recall that $X=\Gen(M,\mathbb{P})$ is the the set of $M$-generics for some poset $\mathbb{P}$.

\begin{definition}
     We say that $f\colon X\to\omega^\omega$ is a {\bf generic  \(\equiv^\mathbb{P}_M\)-invariant uniformization} if and only if
     \begin{enumerate-(i)}
    \item 
    $f(x)$ is $M[x]$-generic for all \(x\in X\); and 
    \item
    \(f\) is a \(\equiv^\mathbb{P}_M\)-invariant uniformization; i.e., \(x\equiv^\mathbb{P}_My\implies f(x)=f(y)\).
     \end{enumerate-(i)}
\end{definition}

\begin{theorem}
\label{thm : invariant function}
If there is a generic \(\equiv^\mathbb{P}_M\)-invariant uniformization \(g\colon X\to\omega^\omega\), then the equivalence relation $\equiv ^\mathbb{P}_M$ is hyperhyperfinite.
\end{theorem}
\begin{proof}
Let \(E\) denote \( \equiv ^\mathbb{P}_M\) and  suppose that $g$ is such a function.

For each $n$ and $y\in\omega^\omega$, let $E^n_y$ be defined as in~\eqref{eq:E_y} above, but with respect to the involutive automorphisms generating $E_{\Gamma_n}$. Namely, take $\{g_k\mid i\in I_n\}$ from above and re-enumerate them as $\{g^n_k\mid k<\omega\}$. Then define $\{E^n_s\mid s\in\omega^{<\omega}\}$, and $E^n_y$ with respect to $f_k(x) \coloneqq \min_{<}\{x, g_k"x\}$, for $k<\omega$. Note that for every $x\mathbin{E_{\Gamma_n}}z$, the set  
$D^n_{x,z}=\{s\in\omega^{<\omega}\mid x\mathbin{E}^n_sz\}$ is a dense set in $M[x]$ because to define it we only need $x,z, <,  \{g^n_k\mid k<\omega\}$, where the latter two are in $M$. (And of course, $M[x]=M[z]$).

Now, define $E_n$ on \(X\) by setting \[x\mathbin{E_n}z\qquad\iff\qquad x\mathbin{E}z\quad\text{and}\quad xE^n_{g(x)} z.\] Then $E_n$ is a hyperfinite relation because \(E_n= \bigcup_k F_{n,k}\) where \(x\mathbin{F}_{n,k} z \iff x\mathbin{E} z \text{ and } x\mathbin{E}^n_{g(x)\restriction k}z\).

Moreover, if  $x\mathbin{E_{\Gamma_n}}z$, then since $g(x)$ is $M[x]=M[z]$-generic and $D^n_{x,z}\in M[x]$ is a dense set, we have that $x\mathbin{E^n_{g(x)}}z $. It follows that \(\equiv^\mathbb{P}_M\) equals \(\bigcup_n E_n\), and for each $n$, we have $E_n\subseteq E_{n+1}$. 
\end{proof}

\begin{corollary}
\label{cor : Random}
    For any model \(M\), there is no generic \(\equiv^\mathbb{B}_M\)-invariant uniformization, when \(\mathbb{B}\) is the random forcing. 
\end{corollary}
\begin{proof}
The equivalence relation \(\equiv^\mathbb{B}_M\) is not Fr\'echet amenable by \cite[Theorem~4.4]{Smy22}.
Since every hyperfinite equivalence relation is amenable, every hyper-hyperfinite equivalence relation is Fr\'echet amenable (e.g., see~Proposition~9.11(v) of~\cite{Kec24}).
    Therefore the corollary is an immediate consequence of Theorem~\ref{thm : invariant function}.
    \end{proof}

It turns out that a somewhat weaker assumption is enough to get the conclusion of  Theorem~\ref{thm : invariant function}. Denote \(E_0(\mathbb{N})\) be the equivalence relation of eventaul equality on \(\omega^\omega\). It is well known that \(E_0(\mathbb{N})\sim_BE_0\), therefore \(E_0(\mathbb{N})\) is hyperfinite. (E.g., see Gao~\cite[Proposition~6.1.2]{Gao}.)

\begin{proposition}
\label{prop : weak uniformization}
Suppose that $f\colon X\to \omega^\omega$ is a Borel function such that,
\begin{enumerate}
\item
$f(x)$ is \(M[x]\)-generic, for all \(x\in X\);
\item
$x\equiv^\mathbb{P}_M y$ implies $ {f(x)}\mathbin{E_0(\mathbb{N})}{f(y)}$.

\end{enumerate}
Then $\equiv^\mathbb{P}_M$ is hyperhyperfinite.

\end{proposition}
\begin{proof}
Let $f$ be such a function. As before, construct $E_{\Gamma_n}$, for $n<\omega$,  such that each $E_{\Gamma_n}$ is generated by an $M$-set of involutions in $Aut^M(\mathbb{P})$, and $E$ is their increasing union. Fix a partition $\{I_n\mid n<\omega\}$ in $M$ of $\omega$ into infinite sets.

For each $n$, let $\{g^n_i\mid i<\omega\}\in M$ be the involutions in $Aut^M(\mathbb{P})$, so that $E_{\Gamma_n}$ is generated by them; and for $s\in\omega^{<\omega}$, we define $x\mathbin{E^{\Gamma_n}_s}z$ iff $g_s(x)=g_s(z)$, where $g_s$ is defined as in the cascade argument from in \cite[Theorem~12.1]{KecMil} described above. 
Similarly, define $$D^{\Gamma_n}_{x,z} = \{s\mid x\mathbin{E^{\Gamma_n}_s}z\}.$$ Then  for all $x,z$ such that $x\mathbin{E_{\Gamma_n}}z$, $D^{\Gamma_n}_{x,z}$ is dense open in \((\omega^{<\omega}, \supseteq)\) and $D^{\Gamma_n}_{x,z}\in M[x]$.

For $n,k<\omega$, define $xF^n_kz$ iff
 \begin{itemize}
 \item
 for all $n'\geq n$, $f(x)\upharpoonright I_{n'} = f(z)\upharpoonright I_{n'}$, and
 \item
 setting $\bar{y} = f(z)\upharpoonright I_n$, we have $x\mathbin{E_{\bar{y}\upharpoonright k}^{\Gamma_n}}z$.
\end{itemize}
Then $\{ F^n_k\mid k<\omega\}$ is an increasing sequence of finite equivalence relations.

Next we check that $\bigcup_{k,n} F^n_k=E$. Suppose that
$x\mathbin{E}z$. Then for some $n\in\omega$, $x\mathbin{E_{\Gamma_n}}z$, and $D^{\Gamma_n}_{x,z}\in M[x]=M[z]$. By increasing $n$ if necessary, we may assume that for all  $n'\geq n$, $f(x)\upharpoonright I_{n'} = f(z)\upharpoonright I_{n'}$. Let  $\bar{y} = f(x)\upharpoonright I_n$. Since $f(x)$ is $M[x]$-generic and $I_n$ is infinite, by properties of the Cohen forcing, we have that  $\bar{y}$ is $M[x]$-generic. So for some $k$, $\bar{y}\upharpoonright k\in D^{\Gamma_n}_{x,z}$. It follows that $x\mathbin{F}^n_kz$.

Finally, we  check that $\bigcup_k F^n_k\subseteq\bigcup_k F^{n+1}_k$:
suppose that for some $k$, $xF^n_kz$. It follows that for all $n'\geq n$, $f(x)\upharpoonright I_{n'} = f(z)\upharpoonright I_{n'}$ and that $x\mathbin{E_{\Gamma_n}}z$. Hence $x\mathbin{E_{\Gamma_{n+1}}}z$. By a similar argument as above we get that for some $k'$, $x\mathbin{F^{n+1}_{k'}}z$.
\end{proof}

\begin{remark}
Edward Hou recently showed that a partial converse to the above proposition is also true \cite{Hou}. Namely, he proved that if $E$ is hyperfinite, then one can find a function as in the assumption of the proposition.
\end{remark}

\begin{proposition}
Suppose that $y\in \omega^\omega$ and $x\in X$ are $M$-mutually generic. (I.e. $y\times x$ is $M$-generic for $\omega^{<\omega}\times\mathbb{P}$.) Suppose also that $\{g_n\mid n<\omega\}\in M[y]$ are Borel involutions, such that $\equiv^\mathbb{P}_M=\bigcup_n \graph(g_n)$.\footnote{As earlier, it is enough to assume that each $g_n(x) = \gamma_n"x$, for some involution $\gamma_n\in Aut^M(\mathbb{P})$, such that $\{g_n\mid n<\omega\}\in M[y]$.} Then $[x]_{\equiv^\mathbb{P}_M} = [x]_{E_y}$. Namely,  $\equiv^\mathbb{P}_M$ restricted to $\Gen(\mathbb{P}, M[y])$ is hyperfinite. 
\end{proposition}
\begin{proof}
As above, denote $\equiv^\mathbb{P}_M$ by \(E\). The inclusion is $[x]_E \supseteq [x]_{E_y}$ is clear. For the nontrivial direction, suppose $z\in X$ is such that $z\mathbin{E}x$. Then $y$ is $M[z]$ generic (as $M[x]=M[z]$), $D_{x,z}\in M[z]$, and so for some $s\in D_{x,z}$, $s\sqsubset y$. It follows that $z\in[x]_{E_s}\subseteq [x]_{E_y}$. 
The hyperfiniteness of $E$ restricted to $\text{Gen}(\mathbb{P}, M[y])$ follows since every  $\text{Gen}(\mathbb{P}, M[y])$ is mutually generic with $y$.
\end{proof}

Next we look at the Cohen forcing $\mathbb{C}=\Add(\omega, 1)$ to add one real. 
Recall that this poset is isomorphic to  $(\omega^{<\omega}, \supseteq)$. Let $h\colon 2^{<\omega}\rightarrow\omega^{<\omega}$ code that isomorphism, $h\in M$. Note that then for any $M'\supseteq M$, $x\in 2^\omega$ is $M'$-generic iff $h"x$ is $M'$-generic. In slight abuse of notation we use $h(x)$ to denote $h"x$.

Let \(\mathrm{evens}, \mathrm{odds}\) denote the sets of even and odd natural numbers respectively. For $x\in 2^\omega$, let  $x_l = x\upharpoonright \mathrm{evens}$
and  $x_r = x\upharpoonright \mathrm{odds}$.
Note that if $x$ is $M$-generic, then both $x_l$ and $x_r$ are also $M$-generic and moreover, they are mutually generic, i.e., $x_l$ is $M[x_r]$-generic and vice versa.
Finally, Let $\theta\colon2^{\mathrm{evens}}\to\omega^\omega$ be some isomorphism in $M$. Note that if $x\in 2^{\mathrm{evens}}$ is $M$-generic, and $y=\theta(x)$, then so is $y$.

Let \(E\) denote \(\equiv_M^{\mathbb{C}}\), By Feldman-Moore, $E$ is induced by a countable group generated by some Borel involutions $\{g_n\mid n<\omega\}$. Let $N\supset M$ be a countable model, such that the
$\{g_n\mid n<\omega\}\in N$.

Define the following  subequivalence relations of $E$ (both defined on $X$):
\begin{itemizenew}
\item 
$x\mathbin{E^*}z$ iff $x_l=z_l$ and $x_r \mathbin{E} z_r$. 
\item
$x \mathbin{E^N} z$ iff $x=z$ or
$x \mathbin{E}z$,  $x_l=z_l$, and $x$ is $N$ -generic.
\end{itemizenew}
\begin{remark}
   If $x$ is $N$-generic and $x Ez$ for some $M$-generic $z$, there is an automorphism in $M$ witnessing that $M[x]=M[z]$. Since $M\subset N$, this automorphism is also in $N$, and so $z$ is $N$-generic. In particular, $E^N$ is symmetric.
\end{remark}
Note that if there is a generic \(\equiv_M^\mathbb{P}\)-invariant uniformization, then $E\leq_B E^*$.

\begin{lemma}
$E^{\text{N}}$ is hyperfinite.
\end{lemma}
\begin{proof}
We first make some observations. First note that for every $x\in X$, $x_l$ and $x_r$ are mutually generic. Suppose now that $xE^N z$, $x\neq z$, and denote $w=x_l=z_l$. Then $M[x] = M[w][x_r]=M[w][z_r]=M[z]$. So $x_r \equiv_{M[w]}^{\mathbb{P}}z_r$. By assumption $x$ and therefore $z$ are also $N$-generic. It follows that $x_r, z_r$ are $N$-generic, and $w$ is $N[x_r] = N[z_r]$-generic.

Define a sequence of equivalence relations $E_n, n<\omega$ as follows. Set 
\[x\mathbin{E}_n z \qquad\iff\qquad x\mathbin{E}^Nz \text{ and }x_r \mathbin{E}_{y\upharpoonright n} z_r,\]
for \(y=\theta(x_l)\).
Then the \(E_n\)'s are increasing finite equivalence relations. To show that their union is $E^N$, suppose $x\mathbin{E}^Nz$. Let   $w=x_l=z_l$ and $y=\theta(w)$. Then the dense set $D_{x_r, z_r}$ is in $N[x_r]$ and by mutual genericity, $y$ is $N[x_r]$-generic. It follows that  $y\upharpoonright n\in D_{x_r, z_r}$ for some $n$, and so $x_r \mathbin{E}_{y\upharpoonright n} z_r$.
\end{proof}

We end this section with some questions, that may be stepping stones for settling whether equivalence of generics is hyperfinite for the Cohen forcing and other posets.

\begin{question}
    For which posets $\mathbb{P}$, is there a generic \(\equiv^\mathbb{P}_M\)-invariant uniformization? Is there such a uniformization for the Cohen poset?
\end{question}

Note that by Corollary~\ref{cor : Random} there is no invariant mutually generic function for the Random forcing.  On the other hand, if  $\equiv_M^{\mathbb{P}}$ is smooth, then one can define such a function:

\begin{proposition}
     Suppose that  $\equiv_M^{\mathbb{P}}$ is smooth. Then there is a  generic \(\equiv^\mathbb{P}_M\)-invariant uniformization.
\end{proposition}
\begin{proof}
    Let $f\colon\Gen(\mathbb{P}, M) \rightarrow \Gen(\mathbb{P}, M)$ be a Borel selector, and let $\langle\dot{D}_n\mid n<\omega\rangle$ enumerate all $M$-names for dense sets, and finally let $\prec$ be some wellorder of $\mathbb{P}$. Define $g\colon \Gen(\mathbb{P}, M) \rightarrow \Gen(\mathbb{P}, M)$ as follows. For each $x\in \Gen(\mathbb{P}, M)$, inductively construct $\{p_n\mid n<\omega\}$, so that each $p_n$ is the $\prec$-least condition such that $p_n\leq p_{n-1}$ and $p_n\in \dot{D}_n[f(x)]$. Then set $g(x) = \{p\mid \exists n\, (p_n\leq p)\}$. Then $g$ is the desired generic \(\equiv^\mathbb{P}_M\)-invariant uniformization.
\end{proof}

\begin{question}
    Suppose that $\equiv_M^{\mathbb{P}}$ is hyperfinite. Must there be a generic $\equiv^\mathbb{P}_M$-invariant uniformization?
\end{question}

\newcommand{\etalchar}[1]{$^{#1}$}

\end{document}